\documentclass[a4paper,11pt]{article}

  \renewcommand\appendix{\par
  \setcounter{section}{0}
  \setcounter{sub
  section}{0}
  \setcounter{figure}{0}
  \setcounter{table}{0}
  \renewcommand\thesection{ Appendix \Alph{section}}
  \renewcommand\thefigure{\Alph{section}\arabic{figure}}
  \renewcommand\thetable{\Alph{section}\arabic{table}}
}

 \usepackage{amsmath, amsthm, amssymb}
\usepackage{amssymb}
\usepackage{graphicx}
\usepackage{algorithmic}
\usepackage{algorithm}

\usepackage{tikz}
\usetikzlibrary{calc}
\usetikzlibrary{patterns}
\tikzstyle{mybox} = [draw=black, fill=white,  thick,
    rectangle, inner sep=10pt, inner ysep=20pt]
\tikzstyle{mybox} = [draw=black, fill=white,  thick,
    rectangle, inner sep=2pt, inner ysep=2pt]

\usepackage[margin=1.2in]{geometry}

\usepackage{setspace}

\newtheorem{corollary}{Corollary}[section]

\newtheorem{lemma}{Lemma}[section]
\newtheorem{theorem}{Theorem}[section]

\newtheorem{remark}{Remark}[section]

\begin{document}

\title{On some topological and combinatorial lower bounds on the chromatic number of Kneser type hypergraphs}
\author{ Soheil Azarpendar\\Amir Jafari}

\maketitle

\begin{abstract}

In this paper, we prove a generalization of a conjecture of Erd\"{o}s, about the chromatic number of certain Kneser-type hypergraphs. For integers $n,k,r,s$ with $n\ge rk$  and $2\le s\le r$, the $r$-uniform general Kneser hypergraph $\mbox{KG}^r_s(n,k)$, has all $k$-subsets of $\{1,\dots,n\}$ as the vertex set and all multi-sets $\{A_1,\dots, A_r\}$ of $k$-subsets with $s$-wise empty intersections as the edge set. The case $r=s=2$, was considers by Kneser \cite{K} in 1955, where he conjectured that its chromatic number is $n-2(k-1)$. This was finally proved by Lov\'asz \cite{L} in 1978. The case $r>2$ and $s=2$, was considered by Erd\"os in 1973,  and he conjectured that its chromatic number is $\left\lceil\frac{n-r(k-1)}{r-1}\right\rceil$. This conjecture was proved by Alon, Frankl and Lov\'asz  \cite{AFL} in 1986. The case where $s>2$, was considered by Sarkaria \cite{S} in 1990, where he claimed to prove a lower bound for its chromatic number which generalized all previous results. Unfortunately, an error was found by Lange and Ziegler \cite{Z'} in 2006 in the induction method of Sarkaria on the number of prime factors of $r$, and Sarkaria's proof only worked when $s$ is less than the smallest prime factor of $r$ or $s=2$. Finally in 2019, Aslam, Chen, Coldren, Frick and Setiabrata \cite{F}
were able to extend this by using methods different from Sarkaria to the case when $r=2^{\alpha_0}p_1^{\alpha_1}\dots p_t^{\alpha_t}$ and $2\le s\le 2^{\alpha_0}(p_1-1)^{\alpha_1}\dots (p_t-1)^{\alpha_t}$. In this paper, by applying the $\mathbb Z_p$-Tucker lemma of Ziegler \cite{Z} and Meunier \cite{M}, we finally prove the general Erd\"{o}s conjecture and prove the claimed result of Sarkaria for any $2\le s\le r$. We also provide another proof of a special case of this result, using methods similar to those of Alon, Frankl, and Lov\'asz \cite{AFL}  and compute the connectivity of certain simplicial complexes that might be of interest in their own right.

 \end{abstract}

\section{Introduction}

For integers $n,k,r,s$ with $n\ge rk$ and $2\le s\le r$, the $r$-uniform hypergraph $\mbox{KG}^r_s(n,k)$, with all $k$-subsets of $[n]:=\{1,\dots,n\}$ as the vertex set and all multi-sets  $\{A_1,\dots, A_r\}$ of $k$-subsets with $s$-wise empty intersection as the edge set, was considered by Sarkaria in \cite{S}. It generalizes the corresponding graphs and hypergraphs of Kneser \cite{K} and Erd\"os \cite{E}. A further generalization was made by Ziegler in \cite{Z}.  For a sequence $S=(s_1,\dots, s_n)$ of integers $0\le s_i\le r$, the $r$-uniform hypergraph $\mbox{KG}^r_S(n,k)$ was defined with all $k$-subsets of $[n]$ as the vertex set and all multi-sets $\{A_1,\dots,A_r\}$ of $k$-subsets that are $S$-disjoint as the edge set.  Here, $S$-disjoint means that for any $1\le i\le n$ the number $1\le j\le r$ with $i\in A_j$ is at most $s_i$. When $s_1=\dots=s_n=s-1$, we recover the Sarkaria's hypergraph $\mbox{KG}^r_s(n,k)$. Let $\bar{n}=s_1+\dots+s_n$. Ziegler claimed that if $\bar{n}\ge kr$ then the chromatic number of this hypergraph $\chi(\mbox{KG}^r_S(n,k))$ \footnote{Ziegler's assumption $s_i<r$ is dropped. If at least $k$ values of $s_i$'s are equal to $r$, then a loop edge $\{A,\dots, A\}$ appears. In this case, we define the chromatic number to be infinity.} is at least $\left\lceil \frac{\bar{n}-r(k-1)}{r-1}\right\rceil$. Unfortunately, the same type of error that existed in Sarkaria \cite{S}. also exited in this proof in the case when $r$ is non-prime, see \cite{Z'} and \cite{LZ}.

For a partition ${\cal P}=\{P_1,\dots, P_l\}$ of $[n]$, Alishahi and Hajiabolhassan \cite{AH} defined an $r$-uniform hypergraph $\mbox{KG}^r(n,k,\cal P)$ with the vertex set of all $k$-subsets $A$ of $[n]$ such that $|A\cap P_j|\le 1$ for all $j=1,\dots, l$ and the edge set of all  $r$-subsets $\{A_1,\dots, A_r\}$  of vertices with pairwise empty intersection.  They studied the chromatic number of this hypergraph when $r=2$. Aslam, Chen, Coldren, Frick and Setiabrata in \cite{F} studied the chromatic number of this hypergraph for $r>2$ under the assumption that $|P_i|\le r$ for all $1\le i\le l$.

The relation between the two hypergraphs $\mbox{KG}_S^r(n,k)$ and $\mbox{KG}^r(n,k,\cal P)$ is as follows. Let $\bar{n}=s_1+\dots+s_n$ and partition $[\bar{n}]$ into ${\cal P}=\{P_1,\dots,P_n\}$ where $P_1$ is the first $s_1$ elements $\{1,\dots, s_1\}$, $P_2$ is the second $s_2$ elements and so on. 
The map $f:[\bar{n}]\rightarrow [n]$ that for $1\le i\le n$, sends the elements of $P_i$ to $i$, defines a homomorphism of $\mbox{KG}(\bar{n},k,\cal P)$ to $\mbox{KG}_S^r(n,k)$ and hence $$\chi(\mbox{KG}_S^r(n,k))\ge \chi(\mbox{KG}^r(\bar{n},k,\cal P)).$$ Our main result is the proof of the following theorem (conjecture 3.5. in \cite{F}).

\begin{theorem}\label{main}

For integers $n,k,r$ with $n\ge rk$ and $r\ge 2$ and a partition ${\cal P}=\{P_1,\dots, P_l\}$ of $[n]$, if $|P_i|\le r$ for all $1\le i\le l$ then the chromatic number of $\mbox{KG}^r(n,k,\cal P)$ is equal to $\left\lceil \frac{n-r(k-1)}{r-1}\right\rceil$.

\end{theorem}

This theorem was proved in the following special cases in \cite{F}. When $r=2^{\alpha_0}$, when $r$ is an odd prime number and $|P_i|\le r-1$. More generally when $r=2^{\alpha_0}p_1^{\alpha_1}\dots p_t^{\alpha_t}$ and $|P_i|\le 2^{\alpha_0}(p_1-1)^{\alpha_1}\dots (p_t-1)^{\alpha_t}$. The method we use to prove this theorem is different and simpler than the method used in \cite{F}. The proof uses the $\mathbb Z_p$-Tucker lemma of Ziegler and Meunier,  see \cite{Z} and \cite{M}.

In particular, a corollary to our main theorem above is the following claimed result of Ziegler (without the extra assumption of $s_i<r$), whose proof had an error.

\begin{corollary}\label{z} If $n,k,r$ and $S=(s_1,\dots, s_n)$ are as above and $\bar{n}=s_1+\dots+s_n\ge kr$ then 
$$\chi(\mbox{KG}^r_S(n,k))\ge \left\lceil \frac{\bar{n}-r(k-1)}{r-1}\right\rceil.$$
\end{corollary}

The proof of this corollary when $r=p$ is a prime number and $s_i<r$, in Ziegler \cite{Z}, was correct and used the $\mathbb Z_p$-Tucker lemma. In the last section of this article, we provide a different proof for this case, where we compute the connectivity of a certain complex whose method might be useful in other places. In this section, some other complexes are defined that might be of interest as well. It is worthwhile to note that here, we allow $s_i=r$, which was excluded by Ziegler.

Erd\"os \cite{E} introduced the following method to derive an upper bound for $\chi(\mbox{KG}_S^r(n,k))$. Assume that $s_1\le s_2\le \dots\le s_n<r$ and $s_1+\dots+s_n\ge kr$. Let $t_1$ be the largest number such that $s_1+\dots+s_{t_1}<rk$. Now let $t_2$ be the largest number such that $s_{(t_1+1)}+\dots +s_{(t_1+t_2)}<r$, and continue this way to find the largest $t_i$ so that 
$$s_{(t_1+\dots+t_{i-1}+1)}+\dots+s_{(t_1+\dots+t_{i-1}+t_i)}<r$$
 If after $l$ steps, the process stops, i.e. $t_1+\dots+t_l=n$, then $\chi(\mbox{KG}_S^r(n,k))\le l$. Here we give a proper coloring as follows. We will color $A$ with color $1$ if $A\subset \{1,\dots, t_1\}$, and with color $i$ (for $i=2,\dots, l$) if it contains at least one of the elements $t_1+\dots+t_{i-1}+1,\dots, t_1+\dots+t_{i-1}+t_i$. To see that this coloring is proper, assume that for an $S$-disjoint multiset $\{A_1,\dots, A_r\}$ of $k$-subsets, all the sets have color $1$, then since on one hand $\sum_{i=1}^r |A_i|=rk$ and on the other hand each element $1\le i\le t_1$ is repeated at most $s_i$ times, this sum is at most $\sum_{i=1}^{t_1}s_i<kr$ which is a contradiction. If all these subsets have color $i\ge 2$, then the number of sets that contain $j$ between $t_1+\dots+t_{i-1}+1$ and $t_1+\dots+t_i$ is at most $s_j$, since each $A_1,\dots, A_r$ must contain at least one of these $j$'s hence the sum of these numbers should be at least $r$, but it is assumed that $\sum s_j$ for these $j$'s is less than $r$, which is a contradiction again.

  If $s_1=\dots=s_n=s<r$ and $sn\ge rk$, this method yields a coloring of $\mbox{KG}^r_S(n,k)$ with $1+\left\lceil{\frac{ns-rk+1}{P s}}\right\rceil$ colors where $P=\lfloor \frac{r-1}{s} \rfloor$. This is because $t_2=\dots=t_{l-1}=P$ and $t_1=\lfloor \frac{kr-1}{s}\rfloor$. This is the content of Lemma 3.1 of \cite{Z}.

So, we can complete the proof of the following corollary of Ziegler \cite{Z} and Sarkaria \cite{S}.

\begin{corollary}
If $s_1=\dots=s_n=s$ and $s$ divides $r-1$ and $ns\ge rk$ then
$$\chi(\mbox{KG}^r_S(n,k))=\left\lceil\frac{ns-r(k-1)}{r-1}\right\rceil.$$
\end{corollary}

\begin{remark}
There are other examples of $S$, where the lower bound and upper bound for the chromatic number of $\mbox{KG}_S^r(n,k)$ coincide. For instance if $s_1=\dots=s_{kr-1}=1$ and $s_{kr}=\dots=s_n=r-1$, we get $\chi(\mbox{KG}_S^r(n,k))=n-rk+2$.
\end{remark}

\section{Proof of the Theorem \ref{main}}

Because of Lemma 3.2 of \cite{F}, the validity of Theorem \ref{main} for $r=r_1$ and $r=r_2$ implies its validity for $r=r_1r_2$. So it is enough to prove Theorem \ref{main} for the case when $r=p$ is a prime number. For this purpose, we follow the method of Meunier in \cite{M} and use the $\mathbb{Z}_p$-Tucker lemma, whose statement, we now recall from \cite{M}.

The classical Tucker lemma was introduced by Tucker in \cite{T} as a combinatorial counterpart for the topological Borsuk-Ulam theorem that asserts there is no continuous map $f$ from the $n$ dimensional sphere $S^n$ to the $(n-1)$ dimensional sphere $S^{n-1}$ that is antipodal i.e. $f(-x)=-f(x)$. It replaces the sphere $S^n$ with the simplicial complex $\mbox{E}_n(\mathbb Z_2)$, which is $(n+1)$-fold joins of the discrete two pointed simplicial complex $\mathbb Z_2$, whose vertices are $\{\pm 1,\dots, \pm(n+1)\}$ and its faces are all subsets that do not contain $i$ and $-i$ together. If we let $sd(X)$ to denote the barry-centric subdivision of a simplicial complex $X$, then the statement of the Tucker lemma in a special case is
\begin{theorem} (Tucker's lemma) There is no antipodal  simplicial map from $sd(\mbox{E}_n(\mathbb Z_2))$ to $\mbox{E}_{n-1}(\mathbb Z_2)$.
\end{theorem}

Dold generalized this lemma to the case when the action of $\mathbb Z_2$ is replaced with the action of $\mathbb Z_p$, see \cite{D}. Similarly one lets $\mbox{E}_n(\mathbb Z_p)$ be the $(n+1)$-fold joins of the discrete simplicial complex $\mathbb Z_p$ with its natural free $\mathbb Z_p$ action. Here the vertices are $\mathbb Z_p\times [n+1]$ where $\mathbb Z_p$ acts on the first component by multiplication (we may let $\mathbb Z_p$ be the multiplicative group of the $p$th roots of unity) and its faces are all subsets $F$ such that $\pi_2:F\rightarrow [n]$ is one to one, where $\pi_2$ is the projection on to the second component. The statement of his theorem in a special case is

\begin{theorem} \label{dold}There is no $\mathbb Z_p$-equivariant simplicial map from $sd(\mbox{E}_n(\mathbb Z_p))$ to $\mbox{E}_{n-1}(\mathbb Z_p)$ or more generally to any simplicial complex $X$ with a free $\mathbb Z_p$ action and of dimension less than $n$.
\end{theorem}

Dold gave a topological proof for this theorem, later Ziegler produced a combinatorial proof in \cite{Z}. The following lemma is called the $\mathbb Z_p$-Tucker lemma and is an immediate corollary to the Dold's theorem above. 

\begin{lemma}\label{tucker}{\textbf{ ($\mathbb Z_p$-Tucker Lemma)}} Let $n, m$ and $\alpha$ be integers and $p$ be a prime number, where $\alpha\le m$. If $\lambda$ is a $\mathbb Z_p$-equivariant map from the non-empty faces of $\mbox{E}_{n-1}(\mathbb Z_p)$ to $\mathbb Z_p\times [m]$ with $\lambda(A)=(\lambda_1(A),\lambda_2(A))\in \mathbb Z_p\times [m]$ that satisfies the following properties,
\begin{enumerate}
\item If $A_1\subseteq A_2$ be non-empty faces of $E_{n-1}(\mathbb Z_p)$ and $\lambda_2(A_1)=\lambda_2(A_2)\le \alpha$ then $\lambda_1(A_1)=\lambda_1(A_2)$.
\item If $A_1\subseteq \dots\subseteq A_p$ be non-empty faces of $E_{n-1}(\mathbb Z_p)$ and $\lambda_2(A_1)=\dots=\lambda_2(A_p)>\alpha$ then $\lambda_1(A_1),\dots, \lambda_1(A_p)$ are not pairwise distinct.
\end{enumerate}
then $$\alpha+(m-\alpha)(p-1)\ge n.$$
\end{lemma}

The proof immediately follows from the Dold's Theorem \ref{dold}, if one notices that the conditions are designed to guarantee the existence of a $\mathbb Z_p$-equivariant simplicial map from $sd(\mbox{E}_{n-1}(\mathbb Z_p))$ to $\mbox{E}_{\alpha-1}(\mathbb Z_p)\star \underbrace{\partial \Delta^{p-1}\star \dots\star \partial \Delta^{p-1}}_{\tiny{m-\alpha}}$, where $\partial\Delta^{p-1}$ is the simplicial complex of all proper subsets of $\mathbb Z_p$ with its free $\mathbb Z_p$ action. Since dimension of the target complex is $\alpha+(m-\alpha)(p-1)-1$, Dold's theorem finishes the proof.
\\
Now, we present our proof of Theorem \ref{main} for the case when $r=p$ is a prime number.
\begin{proof}
Let $t$ be the chromatic number of $\mbox{KG}^p(n,k,{\cal P})$ and $c$ be a proper coloring map from the vertices to $\{1,\dots, t\}$. Also choose an arbitrary complete ordering on the subsets of $[n]$, such that if $|A|<|B|$ then $A<B$. We define a mapping $\lambda$ as in Lemma \ref{tucker} for $m=p(k-1)+t$ and $\alpha=p(k-1)$. Therefore, we conclude $$p(k-1)+t(p-1)\ge n$$ and hence $t\ge \frac{n-p(k-1)}{p-1}$. Since by the standard way of coloring of a Kneser hypergraph $t\le \left\lceil\frac{n-p(k-1)}{p-1}\right\rceil$, the theorem follows.

Assume that ${\cal P}=\{P_1,\dots, P_l\}$ is a partition of $[n]$ with $|P_i|\le p$ for $i=1,\dots, l$. For a non-empty face $A$ of $\mbox{E}_{n-1}(\mathbb Z_p)$
and $i\in \mathbb Z_p$, let 
$$A^i =\{1\le j\le n| (i,j)\in A\}.$$
Choose the maximum subset (with respect to the chosen order)  $B\subseteq A$ such that for all $i\in \mathbb Z_p$, one has $B^i$ is admissible, that is  $|B^i\cap P_j|\le 1$ for all $1\le j\le l$. 
\\
Let $1\le j\le l$ be the smallest number that $\pi_2(B)\cap P_j$ is non-empty.  There is a  unique subset $B'=\{(i_1,j_1),\dots, (i_h,j_h)\}\subseteq B$ with $j_1<\dots<j_h$ such $\pi_2(B')=\pi_2(B)\cap P_j$ . We define 
$$P(A)=\begin{cases}
i_1&\mbox{if}\:\: h=p\\
(i_1\dots i_h)^{h'}&\mbox{if}\: h<p \end{cases}
$$
where $1\le h'<p$ is the unique number that $hh'\equiv 1\mod p$.
With these preliminaries, the definition of the map $\lambda$ is given in two cases.
\\
{\textbf{Case 1:}} If $|B|\le \alpha=p(k-1)$ then define $$\lambda(A)=(P(A),|B|).$$
\\
{\textbf{Case 2:}} If $|B|>p(k-1)$, then there is a  $k$-subset $F\subseteq B^i$ for some $i\in \mathbb Z_p$. Note that by assumption on $B$, one has $|F\cap P_j|\le 1$ for all $1\le j\le l$. Choose the smallest (with respect to the chosen order) such subset, say $F\subseteq B^i$ and define $$\lambda(A)=(i, c(F)+\alpha).$$

It remains to check the conditions of the $\mathbb Z_p$-Tucker lemma. Note that if $B\subseteq A$ is the maximum subset such that $B^i$ is admissible for all $i\in\mathbb Z_p$, then $\omega\cdot B\subseteq \omega\cdot A$ for $\omega\in \mathbb Z_p$, is the maximum subset such that $(\omega\cdot B)^i$ is admiisble for all $i\in {\mathbb Z_p}$. So, we have the same $j$ as in the definition, and $(\omega\cdot B)'=\{(\omega\cdot i_1,j_1),\dots, (\omega\cdot i_h ,j_h)\}$, and hence $P(\omega\cdot A)=\omega\cdot P(A)$. This obvious if $h=p$ and follows from $\omega^{hh'}=\omega$ if $h<p$. This proves that $\lambda$ is a $\mathbb Z_p$-map in the Case 1. In the second case if $F\subseteq B^i$ is the maximum $k$-subset then $F\subseteq (\omega\cdot B)^{\omega\cdot i}$ is the corresponding maximum subset and hence $\lambda(\omega\cdot A)=\omega\cdot \lambda(A)$. 
\\

Now assume $A_1\subseteq A_2$ are non-empty faces of $\mbox{E}_{n-1}(\mathbb Z_p)$ with $\lambda_2(A_1)=\lambda_2(A_2)\le \alpha$. Hence we are in the Case 1. Let $B_1\subseteq A_1$ and $B_2\subseteq A_2$ be maximum subsets with the property that $B_1^i$ and $B_2^i$ are admissible for all $i\in \mathbb Z_p$. Since $\lambda_2(A_1)=\lambda_2(A_2)$ hence $|B_1|=|B_2|$. Since $B_1\subseteq A_2$, also has the property that $B_1^i$ is admissible, hence $|\pi_2(B_1)\cap P_j|=|\pi_2(B_2)\cap P_j|$ for all $1\le j\le l$. It follows that the smallest $1\le j\le l$, needed in the definition is the same for both of them. Hence their corresponding $h$ is also the same. If $h=p$ then since $|P_j|\le p$ then $\pi_2(B_1)\cap P_j=\pi_2(B_2)\cap P_j=P_j$ and hence $B_1'=B_2'$ and it implies that $\lambda_1(A_1)=\lambda_1(A_2)$. If $h<p$, then since $\pi_1(B_1')=\pi_1(B_2')$ then again it follows that $\lambda_1(A_1)=\lambda_1(A_2)$.
\\

Finally, assume that $A_1\subseteq \dots\subseteq A_p$ are non-empty faces of $\mbox{E}_{n-1}(\mathbb Z_p)$ and $\lambda_2(A_1)=\dots=\lambda_2(A_p)>\alpha$. We are then in the Case 2, and we may find admissible $k$-subsets $F_i\subseteq B_i^{\lambda_1(A_i)}$ for $i=1,\dots, p$ with the same color $c(F_1)=\dots=c(F_p)$. If $\lambda_1(A_1),\dots, \lambda_1(A_p)$ are pairwise distinct then since $B_i^{\lambda_1(A_i)}\cap B_j^{\lambda_1(A_j)}=\emptyset$, the subsets $F_1,\dots, F_p$ are pairwise disjoint. This contradicts the properness of the coloring $c$. Hence the conditions are checked and the theorem is proved.
\end{proof}

 \section{Proof of the Corollary \ref{z} when $r$ is a prime number}

For an $r$-uniform hypergraph $G$, Alone, Frankl and Lov\'asz introduced a  simplicial complex $C(G)$ in \cite{AFL} as follows. Its vertex set is all $r$ tuples $(a_1,\dots,a_r)$ of vertices where $\{a_1,\dots, a_r\}$ is an edge. Its faces are all subsets $\{(a_1^i,\dots,a_r^i)\}_{i\in I}$ such that for all choices of not necessarily distinct indices $i_1,\dots, i_r\in I$ the subset $\{a_1^{i_1},\dots, a_r^{i_r}\}$ is an edge. The cyclic group $\mathbb Z_r$ acts on this complex by shifting. It was proved in \cite{AFL} that

\begin{theorem}\label{thm1} If $r$ is a prime number and the complex $C(G)$ is $c$-connected then $$\chi(G)\ge \left\lceil \frac{c+r+1}{r-1}\right\rceil.$$
\end{theorem}

To compute a lower bound for the connectivity of $C(G)$, we work with its maximal nerve. Let us recall its definition. If $C$ is a simplicial complex, its maximal nerve, denoted by $N(C)$ is a simplicial complex whose vertices are maximal faces of $C$ and a subset $\{F_1,\dots,F_m\}$ is a face if and only if $F_1\cap\dots\cap F_m\ne \emptyset$. It is a well-known fact that $C$ and $N(C)$ are homotopy equivalent, see for example \cite{Bo}. Hence $C$ and $N(C)$ have the same connectivity number. The maximal nerve of $C(\mbox{KG}^r_S(n,k))$, and even a slightly more general complex $K_S(n,k_1,\dots, k_r)$, will be explicitly constructed as follows.

Given an integer $r\ge 1$, an integer-valued function $s$ on $[n]$ such that $0\le s(i)\le r$ and an $r$-tuple $\mathbf{k}=(k_1,\dots, k_r)$, where $k_i\ge 0$ define two complexes $K_s(n,\mathbf{k})$ and $C_s(n,\mathbf{k})$ as follow. The vertices of $K_s(n,\mathbf{k})$ are the $r$-tuples $(A_1,\dots, A_r)$ of subsets of $\{1,\dots, n\}$ such that $|A_i|= k_i$ and the $r$-tuple is $s$-disjoint (that is $i$ appears in at most $s(i)$ of $A_1,\dots, A_r$).  The subset $\{(A_1^i,\dots, A_r^i)\}_{i\in I}$ is a face if for all $i_1,\dots, i_r\in I$ the $r$-tuple $(A_1^{i_1},\dots, A_r^{i_r})$ is $s$-disjoint. The vertices of $C_s(n,\mathbf{k})$ are the $r$-tuples $(A_1,\dots, A_r)$ of subsets of $\{1,\dots, n\}$ such that $|A_i|\ge k_i$ and for each $x\in \{1,\dots, n\}$ the number of $1\le i\le r$ that $x\in A_i$ is exactly $s(x)$. The subset $\{(A_1^i,\dots, A_r^i)\}_{i\in I}$ is a face if $|\cap_{i\in I}A_j^i|\ge k_j$ for all $j=1,\dots, r$. If $s(1)+\dots+s(n)< k_1+\dots+k_r$ or $k_i$ is bigger than the number of $x$ with $s(x)>0$, then clearly this complex is empty. The exact condition for which the complex is non-empty is more complicated and we do not need it here. But we need the following result.
\begin{lemma} \label{empty}The complex $C_s(n,k,\dots, k)$, where $k$ is repeated $r$ times is non-empty if and only if $\sum_{i=1}^n s(i)\ge rk$.
\end{lemma}

\begin{proof}
The necessity is obvious. To prove the sufficiency, let $M=\{1\le i \le n| s(i)=r\}$. We prove our claim by induction on $r$. If $r=1$, then the condition of the lemma implies that $|M|\ge k$ so $A_1=M$ is an element of $C_s(n,k)$. More generally if $|M|\ge k$ then $A_1=\dots=A_r=M$ is an element of $C_s(n,k,\dots,k)$, with $k$ repeated $r$ times. So we assume that $|M|<k$. Also assume without loss of generality that $s(1)\ge \dots\ge s(n)$. The condition implies that $s(1)\ge \dots \ge s(k)>0$. Define $s'(i)=s(i)-1$ for $1\le i\le k$ and $s'(i)=s(i)$ for $k<i\le n$. Then $0\le s'(i)\le r-1$ for all $i$, and $\sum_{i=1}^n s'(i)=\sum_{i=1}^n s(i)-k\ge (r-1)k$. So by the induction hypothesis, $C_{s'}(n,k,\dots, k)$ with $k$ repeated $r-1$ times is non-empty. Let $(A_2,\dots, A_r)$ be a vertex of it, and hence $(\{1,\dots, k\}, A_2,\dots, A_r)$ will be a vertex of $C_s(n,k,\dots, k)$ with $k$ repeated $r$ times.
\end{proof}

\begin{lemma}
The maximal nerve of $K_s(n,\mathbf{k})$ is isomorphic to $C_s(n,\mathbf{k})$. 
\end{lemma}
\begin{proof}
 We give a bijection between the vertices of $C_s(n,\mathbf{k})$ with the maximal faces of $K_s(n,\mathbf{k})$ that maps faces to faces. If $(A_1,\dots,A_r)$ is a vertex in $C_s(n,\mathbf{k})$, the collection of all $\{(X_1^i,\dots,X_r^i)\}_{i\in I}$ of all subsets $X_j^i\subseteq A_j$ of size $k_j$, gives a maximal face of $K_s(n,\mathbf{k})$. First, since $(A_1,\dots, A_r)$ is $s$-disjoint any $(X_1,\dots,X_r)$ where $X_i\subseteq A_i$ is also $s$-disjoint. So the collection is a face of $K_s(n,\mathbf{k})$. Second, if the collection were not maximal, one can add $(X_1,\dots,X_r)$ where at least one of $X_i$'s is not a subset of $A_i$. Assume $X_1$ is not a subset of $A_1$, hence there is $x\in X_1\backslash A_1$. Since the number of $2\le i\le r$ that $x\in A_i$ is exactly $s(x)$, one can find $(X_1',\dots,X_r')$ where $X_i'\subseteq A_i$ of size $k_i$ and the number of $2\le i\le r$ that $x\in X_i'$ is exactly $s(x)$, this implies that $(X_1,X_2',\dots,X_r')$ is not $s$-disjoint, which is a contradiction. Conversely if $\{(X_1^i,\dots, X_r^i)\}_{i\in I}$ is a maximal face of $K_s(n,\mathbf{k})$, if we let $A_j=\cup_{i\in I} X_j^i$ for $j=1,\dots, r$, then $(A_1,\dots, A_r)$ is a vertex of $C_s(n,\mathbf{k})$. Since, if $x\in\{1,\dots, n\}$ appears in $l>s(x)$ of $A_i$'s then one can find $i_1<\dots<i_l$ and $j_1<\dots<j_l$ that $x\in X_{j_1}^{i_1},\dots , x\in X_{j_l}^{i_l}$, this contradicts the fact that $\{(X_1^i,\dots, X_r^i)\}_{i\in I}$ is a face. And if if $x\in\{1,\dots, n\}$ appears in $l<s(x)$ of $A_i$'s, by adding $x$ to one of the $X_j^i$ that does not contain it and removing another element from it, we arrive at a set $Z_j^i$ where $(X_1^i,\dots, X_{j-1}^i, Z_j^i,X_{j+1}^i,\dots,X_r^i)$ can be added to the face $\{(X_1^i,\dots, X_r^i)\}_{i\in I}$ without violating the face condition. This contradicts the maximality condition. The fact that this bijection between vertices sends faces to faces is easy and is left to the reader.
\end{proof}

Note that if $\mathbf{k}=(k,\dots, k)$, where  $k$ is repeated $r$ times, then $K_s(n,\mathbf{k})$ is $C(\mbox{KG}_s^r(n,k))$.

To compute the connectivity $C(\mbox{KG}^r_s(n,k))$, we use the following well-known fact, as stated in \cite{AFL}.
\begin{lemma}\label{lem}
If a simplicial complex $X$ is a union of its sub-complexes $X_1,\dots, X_n$ and for a given integer $c$ and for any $1\le l\le n$ and $1\le  i_1<i_2<\dots<i_l\le n$, the intersection $X_{i_1}\cap\dots\cap X_{i_l}$ is $(c-l+1)$-connected then $X$ is $c$-connected.(Recall that $(-1)$-connected means non-empty and any space is $c$-connected for $c\le -2$.)
\end{lemma}

The following theorem, together with Theorem \ref{thm1} will finish the proof of Corollary \ref{z} for the case when $r$ is a prime number. However, the statement of the following theorem does not require that $r$ to be a prime number.
\begin{theorem}
The complex $C_s(n,k,\dots, k)$ where $k$ is repeated $r$ times is $(s(1)+\dots+s(n)-rk-1)$-connected.
\end{theorem}
\begin{proof}
Let $M=\{1\le i\le n| s(i)=r\}$, $\mbox{Supp}(s)=\{1\le i\le n| s(i)>0\}$ and $c=s(1)+\dots+s(n)-kr-1$. We need to show that $C_s(n,k,\dots,k)$, where  $k$ is repeated $r$ times is $c$-connected. If $s(1)+\dots+s(n)<kr$ then it is obvious since by convention any space is $c$-connected for $c\le -2$. Also, if $k\le |M|$ then the complex is contractible  since any subset of vertices will be a face. We, therefore, assume $k>|M|$. The complex $C_s(n,k,\dots, k)$ is isomorphic to $C_{s'}(n-|M|,k-|M|,\dots,k-|M|)$ where we remove the elements of $M$ and the corresponding $s_i$'s. Since both complexes have the same corresponding number $c$, we can assume that $M$ is empty. 

 Let $C^i$ for $i\in \mbox{Supp}(s)$ be the subcomplex of $C_s(n,k,\dots, k)$ of those vertices $(A_1,\dots, A_r)$ that $i\in A_1$. According to lemma \ref{lem} it is enough to show that  for any $I\subseteq \mbox{Supp}(s)$,  $C^I=\cap_{i\in I}C^i$ is $(c-|I|+1)$-connected.

 First, we show that when $|I| \ge k$ then $C^I$ is $(c-|I|+k)$-connected. Since $c-|I|+k\ge -1$ is the only non-trivial case, we may assume $|I|\le c+k+1$.  
Under this assumption, it is easy to show that $C^I$ is non-empty. In fact if we let $s'(i)=s(i)-1$ for $i\in I$ and $s'(i)=s(i)$ otherwise, then
$$\sum_{i=1}^n s'(i)-(r-1)k=c+1-|I|+k\ge 0$$
hence by lemma \ref{empty}, there is $(A_2,\dots, A_r)$ in $C_{s'}(n,k,\dots, k)$ with $k$ repeated $r-1$ times. Now $(I, A_2,\dots, A_r)$ is an element of $C^I$.

 By the following Lemma \ref{lem3}, $C^I$ is homotopy equivalent to $C_{s'}(n,k,\dots, k)$ with $k$ repeated $r-1$ times, hence by induction on $r$, $C^I$ is $\sum s'(i)-(r-1)k-1= (c-|I|+k)$-connected. Since $k>0$, it is at least $(c-|I|+1)$-connected.

  For $l=|I|<k$, we show that $C^I$ is in fact $c$-connected. We use a backward induction on $l$. If $l=k-1$, then consider $C^I$ as the union of $C^{I\cup\{i\}}$ for $i\in \mbox{Supp}(s)\backslash I$. Then any intersection of $l'$ of these sub-complexes is  $(c-(l+l')+k)$-connected, so it is at least  $(c-l'+1)$-connected, and hence by lemma \ref{lem}, $C^I$ is $c$-connected. When $|I|<k-1$, then similarly $C^I$ is the union of $C^{I\cup\{i\}}$  and the intersection of any $l'$ of them by the induction hypothesis, if $l+l'<k$ is $c$-connected and if $l+l'\ge k$ is $(c-(l+l')+k)$-connected, and in both cases it is at least $(c-l'+1)$-connected, so by lemma \ref{lem}, $C^I$ is $c$-connected.

\end{proof}

\begin{lemma}\label{lem3}
If $I$ is a subset of $\mbox{Supp}(s)$ of size at least $k$, then $C^I$ is homotopy equivalent to $C_{s'}(n,k,\dots, k)$ where $k$ is repeated $r-1$ times. Here,  $s'(i)=s(i)-1$ for $i\in I$ and $s'(i)=s(i)$ otherwise.
\end{lemma}
\begin{proof}
We define a retraction $r$ from $C^I$ to the sub-complex of vertices $(A_1,\dots, A_r)$ with $A_1=I$,which is isomorphic to $C_{s'}(n,k,\dots, k)$ where $k$ is repeated $r-1$ times. For this purpose, for a vertex $v=(A_1,\dots,A_r)$ of $C^I$, let $r(v)=(I , A_2',\dots, A_r')$ be a vertex of $C^I$ where $A_2',\dots, A_r'$ are obtained by distributing the elements of $A_1\backslash I$ among them only once, so that each $i$ appear in exactly $s(i)$ subset, for example, let $B_i=(A_1\backslash I)\cap A_1\cap\dots\cap A_{i-1}$ for $i=2,\dots, n$ and let $A_i'=A_i\cup B_i$. Notice that if $\{v_1,\dots, v_m\}$ is a face of $C^I$ then $\{v_1,\dots,v_m, r(v_1),\dots, r(v_m)\}$ is also a face, so $r$ is a simplicial map in a stronger sense.  Also the simple homotopy given by $tx+(1-t)r(x)$ for $0\le t\le 1$ on the
geometric realization of the complex $C^I$ is well-defined and shows that $r$ is a deformation retract. 
\end{proof}

\begin{remark}
It seems to be an interesting problem to understand the homotopy type of the complex $C_s(n,k_1,\dots, k_r)$ in general. We believe that it must be wedge of several spheres.
\end{remark}

{\textbf{Acknowledgement.}} The authors like to thank Hamidreza Daneshpajouh for his careful reading and constructive comments that improved the exposition of the  article greatly. They would also like to thank the reviewers whose suggestions made the text much better.

\end{document}